\documentclass[leqno,twoside, 11pt]{amsart}
\usepackage{amssymb, latexsym, esint}
\usepackage{color}
\usepackage{amsmath,amsthm,amscd,amsfonts,mathrsfs}
\usepackage{enumerate}

\makeatletter
\@namedef{subjclassname@2020}{\textup{2020} Mathematics Subject Classification}
\makeatother

\theoremstyle{plain}
\numberwithin{equation}{section} \numberwithin{figure}{section}

\newtheorem{theorem}{Theorem}[section]
\newtheorem{lemma}[theorem]{Lemma}
\newtheorem{proposition}[theorem]{Proposition}

\newtheorem{definition}[theorem]{Definition}
  \usepackage{epsfig} 
\usepackage{graphicx}  \usepackage{epstopdf}
\theoremstyle{definition}
\newtheorem{remark}[theorem]{Remark}

\newcommand{\R}{{\mathbb R}}

\begin{document}

\title[Equivalence of solutions for the $p(x)$-Laplacian]{Equivalence of solutions for non-homogeneous $p(x)$-Laplace equations}
\author{Mar\'ia Medina and Pablo Ochoa}
\address[Mar\'ia Medina]{Departamento de Matem\'aticas,
Universidad Aut\'onoma de Madrid,
Ciudad Universitaria de Cantoblanco,
28049 Madrid, Spain}
\email{maria.medina@uam.es}
\address{Pablo Ochoa. Facultad de Ingenier\'ia, Universidad Nacional de Cuyo. CONICET\\Parque Gral. San Mart\'in 5500\\
Mendoza, Argentina.}
\email{pablo.ochoa@ingenieria.uncuyo.edu.ar}


\begin{abstract}We establish the equivalence between weak and viscosity solutions for non-homogeneous $p(x)$-Laplace equations with a right-hand side term depending on the spatial variable, the unknown, and its gradient. We employ inf- and sup-convolution techniques to state that viscosity solutions are also  weak solutions, and  comparison principles to prove the converse. The new aspects  of the $p(x)$-Laplacian compared to the constant case are the presence of $\log$-terms  and the lack of the invariance under translations. 

\end{abstract}

	

\maketitle

\noindent {\it To the memory of Ireneo Peral, mentor and friend. Te echamos de menos.}

\section{Introduction}

Let $\Omega$ be a bounded domain of $\R^n$. We will consider non-homogeneous $p(x)$-Laplace equations of the form
\begin{equation}\label{Eq 1}
-\Delta_{p(x)}u=f(x, u, D u) \quad \text{in }\Omega,
\end{equation}where, given a function $p: \Omega \to (1, \infty)$,  $-\Delta_{p(x)}$ is the  $p(x)$-Laplace operator defined as
\begin{equation}\label{defLap}
-\Delta_{p(x)} u := -\textnormal{div}\left( |D u|^{p(x)-2}D u\right).
\end{equation}
For a smooth function $\varphi$ with $D\varphi \neq 0$, we can expand the expression above and write
$$-\Delta_{p(x)} \varphi(x)= -|D\varphi|^{p(x)-2}\left(-\Delta \varphi + \frac{(p(x)-2)}{|D\varphi|^2} \Delta_\infty \varphi \right) -|D\varphi|^{p(x)-2}Dp(x)\cdot D\varphi \log|D\varphi|,$$where
$$\Delta_\infty \varphi:= D^{2}\varphi D\varphi\cdot D\varphi $$is the $\infty$-Laplacian. Consequently,  two fundamental differences exist between $-\Delta_{p(x)}$ and the $p$-Laplacian, with $p$ constant: the fact that this operator is not invariant under translations in $x$ and the presence of $\log$-terms.

Recently, the study of partial differential equations with variable exponents has been motivated by the description of models in electrorheological and thermorheological fluids, image processing, or robotics. As an illustrative example, we discuss the model \cite{CLR} for image restoration. Let us consider an input $I$ that corresponds to shades of gray in a domain  $\Omega \subset \mathbb{R}^{2}$, and  assume that I is made up of the true image $u$ corrupted by the noise, which is additive. Thus, the effect of the noise can be eliminated by smoothing the input,  minimizing
 $$E_1(u):=\int_\Omega |D u(x)|^{2}+ |u(x)-I(x)|^{2}dx.$$Unfortunately, smoothing destroys the small details of the image.  A better approach is the total variation smoothing. Since an edge in the image gives rise to a very large gradient, the level sets around the edge are very distinct, so this method does a good job of preserving edges. This method consists of  minimizing the energy
 $$E_2(u):=\int_\Omega |D u(x)|+ |u(x)-I(x)|^{2}dx.$$However, total variation smoothing preserves edges, but it also creates edges that did not exist in the original image. The suggestion of \cite{CLR} was to ensure total variation smoothing ($p=1$) along edges and Gaussian
smoothing  ($p=2$) in homogeneous regions. Furthermore, it employs anisotropic diffusion ($1 < p < 2$) in
regions which may be piecewise smooth or in which the difference between noise and edges is difficult to
distinguish. Specifically, they proposed to minimize

 $$E(u):=\int_\Omega \phi(x, D u) + (u-I)^{2}dx$$where
 $$\phi(x, \xi):= \left\lbrace
  \begin{array}{l}
       \frac{1}{p(x)}|\xi|^{p(x)}, \quad\textnormal{ if  }  |\xi| \leq \beta, \\|\xi|-C(\beta, p(x)), \quad\textnormal{ if }  |\xi| > \beta,\\
  \end{array}
  \right.$$with $\beta >0$ and $1 \leq p(x) \leq 2$. Observe that
where the gradient is sufficiently large (i.e. likely edges), only  total variation based diffusion will be used. Where
the gradient is close to zero (i.e. homogeneous regions), the model is isotropic. At all other locations,
the filtering is somewhere between Gaussian and total variation based. When  minimizing over $u$ of bounded variation, under Dirichlet conditions, the associated  flow  is
$$u_t - \text{div}\left( \phi_r(x, D u)\right) +2(u-I)=0, \text{ in }\Omega \times [0, T],$$with $u(x, 0)=I(x)$, $u$ satisfying the prescribed boundary conditions. Hence, the model above is directly related  to the study of PDE's with the $p(x)$-Laplacian operator
$$\Delta_{p(x)}u:= \text{div }\left(|D u|^{p(x)-2}D u \right).$$Classical references for existence and regularity of solution for $p(x)$-Laplacian Dirichlet problems are \cite{FZ1, FZ, FZ2}, among others.

In this work we are interested in analyzing the equivalence between weak and viscosity solutions (see Section \ref{defSols} for the precise definitions of these notions) of the problem \eqref{Eq 1} under certain conditions on $f$. The relation among different types of solutions for different operators has been studied by several authors in the last decades. For linear problems, the equivalence between distributional and viscosity solutions was obtained by Ishii in \cite{Is}. Later on, for the homogeneous $p$-Laplace operator (i.e., \eqref{Eq 1} with $f\equiv 0$), the equivalence between weak and viscosity solutions was first obtained in \cite{JLM}, and later on in \cite{JJ} with a different proof. For a source term like the one in \eqref{Eq 1}, depending on all the lower-order terms, this equivalence for the $p$-Laplace equation was given in \cite{MO}, following some ideas from \cite{JJ}. Similar studies have been recently made for non-local operators, see \cite{BM, KKL}.

In the case of the variable $p(x)$-Laplacian, the equivalence for the homogeneous equation  was proved in \cite{JLM}. Related results appear in \cite{Sil} between solutions of homogeneous equations involving the {\it strong} and the {\it normalized} $p(x)$-Laplacian. Up to our knowledge, no results are available in the case $f\not\equiv 0$. Indeed, combining techniques from \cite{JLM, Sil} to deal with the operator, and from \cite{MO} to deal with the function $f$, the goal of this work is to prove the equivalence of weak and viscosity solutions for the general problem \eqref{Eq 1}.

Indeed, let us assume from now on that the exponent $p$ satisfies  
\begin{equation}\label{hipP}
p \in \mathcal{C}^{1}(\overline{\Omega}),\quad 1 < p^- \leq p^+ < \infty,\quad \mbox{where}\quad p^-:=\min_{x \in \overline{\Omega}}p(x),\; p^+:=\max_{x \in \overline{\Omega}}p(x).
\end{equation}
Hence, our first result is the following:
\begin{theorem}\label{ViscToWeak}
Let $p$ satisfy \eqref{hipP}. Assume that $f=f(x,t,\eta)$ is uniformly continuous in $\Omega\times\R\times\R^n$, non increasing in $t$, Lipschitz continuous in $\eta$, and satisfies the growth condition 
\begin{equation}\label{growth}
|f(x,t,\eta)|\leq \gamma (|t|)|\eta|^{{p(x)-1}}+\phi(x),
\end{equation}
where $\gamma\geq 0$ is continuous and $\phi\in L^\infty_{loc}(\Omega)$. Thus, if  $u$  is a locally Lipschitz viscosity supersolution of \eqref{Eq 1} then it is a weak supersolution of the problem.
\end{theorem}
It is worth to point out that the regularity assumption on $u$ derives from a technical restriction on $p$. Indeed, if $p^{+}< 2$ it is enough to ask $u \in \mathcal{C}(\Omega)$ in the theorem (see Remark \ref{cont u}).

The proof of Theorem \ref{ViscToWeak} relies on the approximation by the so called inf-convolutions (see Section \eqref{infConvolutions}). Roughly speaking, we will regularize $u$ by some functions $u_\varepsilon$, that will satisfy a related problem in weak sense, and we will pass to the limit here. This idea was first used in \cite{JJ} for the constant $p$-Laplacian in the homogeneous case, and then in more general settings in \cite{BM, MO, Sil}.

The reverse statement, weak solutions being viscosity, is strongly connected with comparison arguments, and a new class of functions needs to be considered.

\begin{definition}Let $u$ be a weak supersolution to \eqref{Eq 1} in $D\subseteq \Omega$. We say that $(u, f)$ satisfies the comparison principle property {\em (CPP)} in $D$ if for every weak subsolution $v$ of \eqref{Eq 1} such that $u\geq v$ a.e. in $\partial D$ we have $u\geq v$ a.e. in $D$.
\end{definition}
In particular we will see that, for those functions satisfying this property, weak solutions are indeed viscosity solutions.

\begin{theorem}\label{WeakToVisc}
Let $p$ satisfy \eqref{hipP}. Assume $u$ is a continuous weak supersolution of \eqref{Eq 1} and $f=f(x,t,\eta)$ is continuous in $\Omega\times\R\times\R^n$ and Lipschitz continuous in $\eta$. If $(CPP)$ holds then $u$ is a viscosity supersolution of \eqref{Eq 1}.
\end{theorem}

In Section \ref{WtoV}, apart from this theorem, we prove a comparison principle for the general equation \eqref{Eq 1}, which has interest in itself.

Applications of the equivalence between viscosity and weak solutions can be found in  \cite{JL, Sil} to removability of sets and Rad\'o type theorems. Also, the equivalence has been recently used in  free-boundary problems (see \cite{BLO, FL}).

The paper is organized as follows: in Section \ref{preliminaries} we give an introduction into the theory of Sobolev spaces with variable exponents, we introduce the notions of viscosity and weak solutions in this context, and the definition and main properties of the infimal convolutions. Section \ref{VtoW} is devoted to the proof of Theorem \ref{ViscToWeak}, that is, to see that viscosity solutions of \eqref{Eq 1} are also weak solutions. In Section \ref{WtoV} we prove Theorem \ref{WeakToVisc} (that weak solutions are viscosity solutions) and a general comparison principle for equation \eqref{Eq 1}.

\section{Preliminaries}\label{preliminaries}

\subsection{Variable exponent spaces}

In this section we introduce basic definitions and preliminary results concerning the spaces of variable exponent and the related theory of differential equations. Let
$$\mathcal{C}_{+}(\overline{\Omega}):=\left\lbrace p \in \mathcal{C}(\overline{\Omega}): p(x)>1 \,\,\text{ for any }\,\, x \in \overline{\Omega}\right\rbrace$$
and 
$$p^{-}:=\min_{\overline{\Omega}}p(\cdot), \quad  p^{+}:=\max_{\overline{\Omega}}p(\cdot).$$
Assume that $p$ belongs to $\mathcal{C}_{+}(\overline{\Omega})$ and satisfies the following log-H\"{o}lder condition: there exists $C>0$ so that
\begin{equation}\label{assumpt H}
|p(x)-p(y)|\leq C\frac{1}{|\log|x-y||}, \quad \text{for all } \;x, y \in \Omega,\; x\neq y.
\end{equation}
We define the Lebesgue variable exponent  space as
$$L^{p(\cdot)}(\Omega):=\left\lbrace u: \Omega \to \mathbb{R}: u \text{ is measurable and }\int_\Omega |u(x)|^{p(x)}\,dx < \infty\right\rbrace,$$
and we denote by $L^{p'(\cdot)}(\Omega)$ the conjugate space of $L^{p(\cdot)}(\Omega)$, where
$$\frac{1}{p(\cdot)}+ \frac{1}{p'(\cdot)}=1.$$
Consider also the Luxemburg norm
$$\|u\|_{L^{p(\cdot)}}:=\inf \left\lbrace \lambda >0: \int_\Omega \Big\vert \dfrac{u(x)}{\lambda}\Big\vert ^{p(x)}\,dx  \leq 1\right\rbrace.$$
Then the following results, that can be found in \cite{DHHR}, hold.

\begin{theorem}[H\"{o}lder's inequality]
The space $(L^{p(\cdot)}(\Omega), \|\cdot
\|_{L^{p(\cdot)}(\Omega)})$ is a separable, uniform convex Banach
space. Furthermore, if $u \in L^{p(\cdot)}(\Omega)$ and $v \in
L^{p'(\cdot)}(\Omega)$, then
$$\Big\vert \int_\Omega u v \,dx\Big\vert   \leq \left(\frac{1}{p^{-}}+\frac{1}{(p')^{-}} \right)\|u\|_{L^{p(\cdot)}(\Omega)}\|v\|_{L^{p'(\cdot)}(\Omega)}.$$
\end{theorem}
The next proposition states the relation between norms and integrals of $p(x)$-th power. 
\begin{proposition}\label{properties modulo} Let
$$\rho(u):=\int_\Omega |u|^{p(x)}\,dx, \quad u \in L^{p(\cdot)}(\Omega),$$
be the convex modular. Then the following assertions hold:
\begin{itemize}
\item[(i)] $\|u\|_{L^{p(\cdot)}(\Omega)} <1$ (resp. $=1, >1$) if and only if $\rho(u)< 1$ (resp. $=1, >1$);
\item[(ii)] $\|u\|_{L^{p(\cdot)}(\Omega)} >1$ implies $\|u\|^{p^{-}}_{L^{p(\cdot)}(\Omega)} \leq \rho(u) \leq \|u\|^{p^{+}}_{L^{p(\cdot)}(\Omega)}$, and $\|u\|_{L^{p(\cdot)}(\Omega)} <1$ implies $\|u\|^{p^{+}}_{L^{p(\cdot)}(\Omega)} \leq \rho(u) \leq \|u\|^{p^{-}}_{L^{p(\cdot)}(\Omega)} $;
\item[(iii)] $\|u\|_{L^{p(\cdot)}(\Omega)}  \to 0$ if and only if $\rho(u)\to 0$, and $\|u\|_{L^{p(\cdot)}(\Omega)}  \to \infty$ if and only if $\rho(u)\to \infty$.
\end{itemize}
\end{proposition}
The following result allows us to relate the norms of different Lebesgue variable  exponent spaces (see \cite{ER} for a proof).
\begin{lemma}\label{product}
Suppose that $p, q \in \mathcal{C}_{+}(\overline{\Omega})$. Let $f \in L^{q(\cdot)p(\cdot)}(\Omega)$. Then
\begin{itemize}
\item[(i)] $\|f\|^{p^{+}}_{L^{p(\cdot)q(\cdot)}(\Omega)} \leq \|f^{p(\cdot)}\|_{L^{q(\cdot)}(\Omega)} \leq \|f\|^{p^{-}}_{L^{p(\cdot)q(\cdot)}(\Omega)} $ if $\|f\|_{L^{p(\cdot)q(\cdot)}(\Omega)} \leq 1$;
\item[(ii)] $\|f\|^{p^{-}}_{L^{p(\cdot)q(\cdot)}(\Omega)} \leq \|f^{p(\cdot)}\|_{L^{q(\cdot)}(\Omega)} \leq \|f\|^{p^{+}}_{L^{p(\cdot)q(\cdot)}(\Omega)} $ if $\|f\|_{L^{p(\cdot)q(\cdot)}(\Omega)} \geq 1$.
\end{itemize}
\end{lemma}
Let us denote the distributional gradiente by $D u$. Then we can define the variable Sobolev space $W^{1, p(\cdot)}(\Omega)$ as
$$W^{1, p(\cdot)}(\Omega):=\left\lbrace u \in L^{p(\cdot)}(\Omega): |D u|\in L^{p(\cdot)}(\Omega)\right\rbrace,$$
equipped with the norm
$$\|u\|_{W^{1, p(\cdot)}(\Omega)}:=\|u\|_{L^{p(\cdot)}(\Omega)}+\|D u\|_{L^{p(\cdot)}(\Omega)},$$
and we denote by $W_0^{1, p(\cdot)}(\Omega)$ the closure of $\mathcal{C}_0^{\infty}(\Omega)$ in $W^{1,
p(\cdot)}(\Omega)$. Notice that, due to the log-H\"{o}lder condition \eqref{assumpt H}, $\mathcal{C}_0^{\infty}(\Omega)$ is dense in $W^{1, p(\cdot)}(\Omega)$. The following 
 Embedding Theorem can be proved (see for instance \cite{FZ4}).
\begin{theorem}\label{Sob emb}
If $p^{+}< n$, then
$$
0<S(p(\cdot),q(\cdot),\Omega) := \inf_{v\in W^{1,p(\cdot)}_0(\Omega)}
\frac{\|D v\|_{L^{p(\cdot)}(\Omega)}}{\|v\|_{L^{q(\cdot)}(\Omega)}},
$$
for all
$$
1\leq q(\cdot)\le p^*(\cdot) = \frac{np(\cdot)}{n-p(\cdot)}.
$$
\end{theorem}
\begin{remark}
The $q(\cdot)$ exponent has to be uniformly subcritical, i.e.
$\inf_\Omega(p^*(\cdot) - q(\cdot)) > 0$, to enssure that
$W^{1,p(\cdot)}_0(\Omega)\hookrightarrow L^{q(\cdot)}(\Omega)$ is
compact.
\end{remark}
Let $X=W_0^{1, p(\cdot)}(\Omega)$. Recalling definition \eqref{defLap}, the operator $-\Delta_{p(x)}$ can be seen as the weak derivative of the functional $J:X\to \mathbb{R}$,
$$J(u):=\int_\Omega\frac{1}{p(x)}|D u|^{p(x)}\,dx,$$
in the sense that if $L:=J': X \to X^{*}$ then
$$(L(u), v)=\int_\Omega |D u|^{p(x)-2}D u D v\,dx, \quad u, v \in X.$$
We also recall the following properties from \cite{FZ2}.
\begin{theorem}\label{properties}Let $X=W_0^{1, p(\cdot)}(\Omega)$. Then:
\begin{itemize}
\item[(i)] $L: X\to X^{*}$ is continuous, bounded and strictly monotone;
\item[(ii)] $L$ is a mapping of type $(S_{+})$, that is, if $u_n \rightharpoonup u$ in $X$ and
$$\limsup_{n\to \infty}(L(u_n)-L(u), u_n-u)\leq 0$$then $u_n \to u$ in $X$;
\item[(iii)] $L$ is a homeomorphism.
\end{itemize}
\end{theorem}



\subsection{Notions of solutions}\label{defSols}
Considering the variable Sobolev spaces defined before, we can already introduce the notion of weak solution.

\begin{definition}We say that $u\in W^{1, p(x)}(\Omega)$ is a weak supersolution of \eqref{Eq 1} if for any non-negative $\varphi \in \mathcal{C}^{\infty}_0(\Omega)$ there holds
$$\int_\Omega |Du|^{p(x)-2}Du\cdot D\varphi\,dx \geq \int_\Omega f(x, u, Du)\varphi\, dx.$$
Likewise, we say that $u \in W^{1, p(x)}(\Omega)$ is a weak subsolution of \eqref{Eq 1} if $$\int_\Omega |Du|^{p(x)-2}Du\cdot D\varphi\,dx \leq \int_\Omega f(x, u, Du)\varphi\, dx$$for any non-negative $\varphi \in \mathcal{C}^{\infty}_0(\Omega)$. 

Finally, $u \in W^{1, p(x)}(\Omega)$ is a weak solution to \eqref{Eq 1} if it is a weak sub- and supersolution.
\end{definition}

Denote by $S^n$ the set of symmetric $n\times n$ matrices. In order to introduce the concept of viscosity solution, let us recall the definition of jets. 
\begin{definition}
The superjet $J^{2, -}u(x)$ of a function $u: \Omega \to \mathbb{R}$ at $x \in \Omega$ is defined as the set of pairs $(\eta, X) \in \left(\mathbb{R}^{n}\setminus \left\lbrace 0 \right\rbrace\right) \times S^{n} $ satisfying
$$u(y)\geq u(x) +\eta\cdot(y-x)+ \dfrac{1}{2}X(y-x)\cdot (y-x) + o(|x-y|^{2})$$as $y \to x$. The closure of a superjet is denoted by $\overline{J}^{2, -}u(x)$ and it is defined as the set of pairs $(\eta, X) \in \mathbb{R}^{n} \times S^{n} $ for which there exists a sequence $(\eta_i, X_i) \in J^{2, -}u(x_i)$, with $x_i \in \Omega$ so that 
$$(x_i, \eta_i, X_i) \to (x, \eta, X) \quad \text{ as }i \to \infty.$$The subjet $J^{2, +}u(x)$ and its closure $\overline{J}^{2, +}u(x)$ are defined in a similar fashion. 
\end{definition}
Observe that the operator can be written as
$$\Delta_{p(x)}\varphi (x)= \textnormal{tr}\left( A(x, D\varphi(x))D^{2}\varphi(x)\right)+ B(x, D\varphi(x)),$$
where
$$A(x, \xi):= |\xi|^{p(x)-2}\left(I+(p(x)-2)\dfrac{\xi}{|\xi|}\otimes \dfrac{\xi}{|\xi|}\right),$$
and
$$B(x, \xi):= |\xi|^{p(x)-2}\log|\xi| \xi \cdot Dp(x).$$
We can now precise the notion of viscosity solution.
\begin{definition}
A lower semicontinuous function $u: \Omega \to \mathbb{R}$ is a viscosity supersolution of \eqref{Eq 1} if for any $(\eta, X) \in J^{2, -}u(x)$ there holds
$$-\text{tr}\left(A(x, X) \right) -B(x, \eta) \geq f(x,u(x),\eta).$$
Similarly, an upper semicontinuous function $u: \Omega \to \mathbb{R}$ is a viscosity subsolution of \eqref{Eq 1} if
$$-\text{tr}\left(A(x, X) \right) -B(x, \eta) \leq f(x,u(x),\eta),$$for all $(\eta, X) \in J^{2, +}u(x)$. Finally, a viscosity solution is a continuous function which is a  viscosity sub- and a supersolution. 
\end{definition}

Observe that we do not require anything at jets of the form  $(0, X)$ or if $J^{2, -}u(x)=\emptyset$. Moreover, the above definition of viscosity supersolution is equivalently given if we replace the superjet by its closure or if we take for jets $(\eta, X)$ pairs of the form  $(D\varphi(x), D^{2}\varphi (x)) \in \left(\mathbb{R}^{n}\setminus \left\lbrace 0 \right\rbrace\right) \times S^{n}$, where $\varphi$ is smooth and touches $u$ from below at $x$.

\subsection{Infimal convolutions}\label{infConvolutions}
A standard smoothing operator in the theory of viscosity solutions is the infimal convolution.
    
\begin{definition}\label{definfconv}
	Given $\varepsilon>0$ and $q\geq 2$ we define the infimal convolution of a function 
	$u\colon \Omega \to \mathbb{R}$ as 
	
	$$u_\varepsilon(x): = 
	\inf_{y\in\Omega}\left(u(y)+\frac{|x-y|^q}{q\varepsilon^{q-1}}\right),\qquad x \in \Omega.$$
\end{definition}
    The infimal convolution will be one of the main tools to prove that
    viscosity solutions are weak solutions. For the next result see for instance \cite{JJ,Sil}
    and the references therein.

\begin{lemma}\label{propinfconv}
	Let $u$ be a bounded and lower semicontinuous function in $\Omega$.
	Then: 
	\begin{enumerate}[(i)]
		\item There exists $r(\varepsilon)>0$ such that 
		\[
		    u_\varepsilon(x) = \inf_{y\in B_{r(\varepsilon)}(x)}\left(u(y)+\frac{|x-y|^q}{q\varepsilon^{q-1}}\right),
		\]
		where $r(\varepsilon)\to 0$ as $\varepsilon\to 0$.
		
		\item The sequence $\{u_\varepsilon\}_{\varepsilon>0}$ is increasing as $\varepsilon\to 0$ and $u_\varepsilon\to u$ pointwise in $\Omega$.
		\item $u_\varepsilon$ is locally Lipschitz and twice differentiable a.e. Actually, for almost every $x,y\in \Omega$,
		\[
		    u_\varepsilon(y) = u_\varepsilon(x)+D u_\varepsilon(x)\cdot(x-y)+\frac{1}{2}D^2u_\varepsilon(x)(x-y)^2
		    +o(|x-y|^2).
		\]

		\item $u_\varepsilon$ is semiconcave, that is, there exists a constant $C=C(q,\varepsilon,\text{osc}(u))>0$ such that the function $x\mapsto u_\varepsilon(x)-C|x|^2$ is concave. In particular 
		\[
	    	D^2u_\varepsilon(x)\leq 2CI,\quad \text{a.e. }\, x\in\Omega,
	    \]
		where $I$ is the identity matrix.
		\item The set $Y_\varepsilon(x) :=\left\{y\in     B_{r(\varepsilon)}(x)\colon u_\varepsilon(x)=u(y)+\frac{|x-y|^q}{q\varepsilon^{q-1}} \right\}$ is non empty and closed for every $x\in\Omega$.
		
		\item If $x \in \Omega_{r(\varepsilon)}:=\left\lbrace x \in \Omega: \text{ dist }(x, \partial \Omega) > r(\varepsilon) \right\rbrace$, then there exists $x_\varepsilon \in B_{r(\varepsilon)}$ such that
		$$u_\varepsilon(x)= u(x_\varepsilon)+ \frac{|x-x_\varepsilon|^q}{q\varepsilon^{q-1}}.$$
		\item If $(\eta, X) \in J^{2, -}u_\varepsilon(x)$ with $x \in \Omega_{r(\varepsilon)}$, then
		$$\eta= \dfrac{(x-x_\varepsilon)}{\varepsilon^{q-1}}|x-x_\varepsilon|^{q-2}\quad \mbox{and}\quad X \leq \frac{q-1}{\varepsilon}|\eta|^{\frac{q-2}{q-1}}I.$$
	\end{enumerate}
\end{lemma}

\begin{remark}For later purposes (see the proof of Lemma \ref{eq for u epsilon}) we will choose 
\begin{equation}\label{q}
q\geq 2 \quad\mbox{ such that }\quad p^--2+\frac{q-2}{q-1}\geq 0.
\end{equation}

\end{remark}

\section{Viscosity solutions are weak solutions: proof of Theorem \ref{ViscToWeak}}\label{VtoW}
Let us consider the inf-convolution $u_\varepsilon$ given by Definition \ref{definfconv}. We can summarize the strategy to prove Theorem \ref{ViscToWeak} in several steps. Assuming that $u$ is a viscosity supersolution, we will identify what problem is satisfied  by $u_\varepsilon$ in a pointwise sense, and later on in a weak sense. We will finish from here by passing to the limit in $\varepsilon$, obtaining the weak problem satisfied by $u$. 

We thus start by identifying the problem fulfilled by $u_\varepsilon$.

\begin{lemma}\label{lemma lips} Assume $p$ satisfies \eqref{hipP}. Let $u:\Omega \to \R$ locally Lipschitz, and let
$f = f (x, t, \eta)$ be continuous in $\Omega \times\R \times \R^n$ and non increasing in $t$. If $u$ is a viscosity supersolution of \eqref{Eq 1} then 
\begin{equation}\label{eq for inf conv}
\Delta_{p(x)}u_\varepsilon (x)\geq f_\varepsilon(x, u_\varepsilon(x), Du_\varepsilon(x)) + E(\varepsilon) \qquad \mbox{a.e. in }\Omega_{r(\varepsilon)},
\end{equation}
where
$$f_\varepsilon(x, s, \eta): = \inf_{y \in B_{r(\varepsilon)}(x)}f(y, s, \eta),$$and $E(\varepsilon) \to 0$ as $\varepsilon \to 0^{+}$. Here $E(\varepsilon)$ depends only on $p$, $q$ and $\varepsilon$.
\end{lemma}
Notice that, differently from the constant case (see \cite[Lemma 3.3]{MO}), when we identify the problem satisfied by $u_\varepsilon$ in a pointwise sense, an error term $E(\varepsilon)$ arises. We expect it that to disappear when passing to the limit in the final step. To prove this lemma we borrow some computations from \cite[proof of Proposition 6.1]{JLP}, where they are used to prove a comparison-type result.
\begin{proof}
Fix $x \in \Omega_{r(\varepsilon)}$ and let $(\eta, Z)\in J^{2, -}u_\varepsilon(x)$, with $\eta \neq 0$. Then, by Lemma \ref{propinfconv}, there is $x_\varepsilon \in B_{r(\varepsilon)}(x)$ such that
\begin{equation}\label{choice x ep}
u_\varepsilon(x)=u(x_\varepsilon)+\dfrac{|x_\varepsilon-x|^{q}}{q\varepsilon^{q-1}}
\quad\mbox{ and }\quad\eta = \dfrac{(x_\varepsilon-x)}{\varepsilon^{q-1}}|x_\varepsilon-x|^{q-2}.\end{equation}
Let $\varphi \in \mathcal{C}^{2}(\mathbb{R}^{n})$ such that $\varphi$ touches $u_\varepsilon$ from below at $x$ and
$$D\varphi(x)=\eta, \quad D^{2}\varphi(x)=Z.$$Then, by definition of $u_\varepsilon$, 
\begin{equation}\label{min}
u(y)-\varphi(z)+ \dfrac{|y-z|^{q}}{q\varepsilon^{q-1}} \geq u_\varepsilon(z)-\varphi(z) \geq 0,
\end{equation}for all $y, z \in \Omega_{r(\varepsilon)}$. Since by \eqref{choice x ep} we have
$$u(x_\varepsilon)= \varphi(x)-\dfrac{|x_\varepsilon-x|^{q}}{q\varepsilon^{q-1}},$$it follows from \eqref{min} that 
$$u(y)-\varphi(z)+ \dfrac{|y-z|^{q}}{q\varepsilon^{q-1}} $$has a minimum at $(x_\varepsilon, x)$. Thus,
$$-u(y)+\varphi(z)- \dfrac{|y-z|^{q}}{q\varepsilon^{q-1}} $$attains its maximum over $\Omega_{r(\varepsilon)}\times \Omega_{r(\varepsilon)}$ at $(x_\varepsilon, x)$. Let us consider
$$\Phi(y, z):= \dfrac{|y-z|^{q}}{q\varepsilon^{q-1}}.$$
By the Maximum Principle for semicontinuous functions (see \cite[Theorem 3.2]{CIL}), there exist symmetric matrices  $(Y, Z)$ such that
$$(\eta, Y)\in \overline{J}^{2, -}u(x_\varepsilon), \quad (\eta, Z) \in \overline{J}^{2, +}\varphi(x),$$and
\begin{equation}\label{ineq matri}
\begin{pmatrix}
-Y & 0 \\ 
0 & Z
\end{pmatrix} \leq D^{2}\Phi(x_\varepsilon, x) + \varepsilon^{q-1}\left(D^{2}\Phi(x_\varepsilon, x) \right)^{2},
\end{equation}with
\begin{equation*}\begin{split}
D^{2}\Phi(x_\varepsilon, x)= &\,\varepsilon^{1-q}|x_\varepsilon-x|^{q-4}\left[|x_\varepsilon-x|^2\begin{pmatrix}
I & -I \\ 
-I & I
\end{pmatrix}  \right.\\
&\left.+ (q-2)\begin{pmatrix}
(x_\varepsilon-x)\otimes(x_\varepsilon-x)  & -(x_\varepsilon-x)\otimes (x_\varepsilon-x)\\ 
-(x_\varepsilon-x)\otimes (x_\varepsilon-x) & (x_\varepsilon-x)\otimes(x_\varepsilon-x)
\end{pmatrix}  \right].
\end{split}\end{equation*}
Inequality \eqref{ineq matri} implies that, for any $\xi, \eta \in \mathbb{R}^{n}$,
\begin{equation}\label{ineq matrices vec}
Z\xi\cdot \xi - Y\eta\cdot \eta \leq \varepsilon^{1-q}\left[(q-1)|x_\varepsilon-x|^{q-2} +2(q-1)^{2}|x_\varepsilon-x|^{2(q-2)} \right]|\eta-\xi|^{2}.
\end{equation}By the equivalence of the definition of viscosity solutions between tests functions and the closure of jets for continuous operators (recall $\eta \neq 0$), we deduce
\begin{equation}\label{f}
\begin{split}
f(x_\varepsilon, &u(x_\varepsilon), \eta) \leq -\text{tr}\left( A(x_\varepsilon, \eta)Y\right)-B(x_\varepsilon, \eta) \\ & = \text{tr}\left( A(x, \eta)Z\right)- \text{tr}\left( A(x_\varepsilon, \eta)Y\right) -\text{tr}\left( A(x, \eta)Z\right) +B(x, \eta)-B(x_\varepsilon, \eta) - B(x, \eta).
\end{split}
\end{equation}Observe that since $\eta \neq 0$, $A(\cdot, \eta)$ is symmetric and positive definite, and hence the square root $A(x, \eta)^{1/2}$ exists and is symmetric. We define
$$A(x)^{1/2}:=A(x, \eta)^{1/2} \quad \text{and} \quad A(x_\varepsilon)^{1/2}:=A(x_\varepsilon, \eta)^{1/2}.$$Now,
\begin{equation}\label{trace}
\begin{split}
\text{tr}\left( A(x, \eta)Z\right) &= \text{tr}\left( A(x)^{1/2}A(x)^{1/2}Z\right)  = \sum_{k=1}^{n} ZA_k(x)^{1/2}\cdot A_k(x)^{1/2},
\end{split}
\end{equation}where $A_k(\cdot)^{1/2}$ is the k-th column of $A(\cdot)^{1/2}$. Hence, \eqref{ineq matrices vec}, \eqref{f}, and \eqref{trace} give
\begin{equation}\label{new est f}
\begin{split}
f(x_\varepsilon, u(x_\varepsilon), \eta) \leq &\,\sum_{k=1}^{n} ZA_k(x)^{1/2}\cdot A_k(x)^{1/2} - \sum_{k=1}^{n} YA_k(x_\varepsilon)^{1/2}\cdot A_k(x_\varepsilon)^{1/2} -\text{tr}\left( A(x, \eta)Z\right) \\ & \,  +B(x, \eta)-B(x_\varepsilon, \eta) - B(x, \eta) \\ \leq &\, C\varepsilon^{1-q}|x_\varepsilon-x|^{q-2}\|A(x)^{1/2}-A(x_\varepsilon)^{1/2}\|_2^{2} + B(x, \eta)-B(x_\varepsilon, \eta) \\ & \,  -\text{tr}\left( A(x, \eta)Z\right)- B(x, \eta). 
\end{split}
\end{equation}
Proceeding as in \cite[proof of Proposition 6.1]{JLP}, it can be seen that 
\begin{equation}\label{est norm diff A}
\|A(x)^{1/2}-A(x_\varepsilon)^{1/2}\|_2^{2}  \leq \dfrac{\|A(x)-A(x_\varepsilon)\|_2^{2}}{(\lambda_{min}(A(x))+\lambda_{min}(A(x_\varepsilon)))^{2}}.
\end{equation}
and, using that $p\in \mathcal{C}^1$, 
\begin{equation}\label{diff B}
\begin{split}
B(x, \eta)-B(x_\varepsilon, \eta) &  \leq  |\eta|^{p(x)-1}|\log|\eta|| |Dp(x)-Dp(x_\varepsilon)|  +C|\eta|^{s-1}\log^{2}|\eta| |p(x)-p(x_\varepsilon)|.
\end{split}
\end{equation}
for some $s$ in the interval connecting $p(x)$ and $p(x_\varepsilon)$.
Furthermore,
\begin{equation}\label{est A}
\begin{split}
\|A(x, \eta)-A(x_\varepsilon, \eta)\|_2 & \leq C\left( (p^{+}+1)|\log|\eta|||\eta|^{s-2} + |\eta|^{p(x_\varepsilon)-2}\right)|x-x_\varepsilon|,
\end{split}
\end{equation}
and
\begin{equation}\label{est lamb min}
\begin{split}
\lambda_{min}\left( A(x)^{1/2}\right) &= \left(\min_{|\xi|=1}A(x, \eta)\xi \cdot \xi \right)^{1/2}  \geq \min\left\lbrace 1, \sqrt{p(x)-1}\right\rbrace |\eta|^{\dfrac{p(x)-2}{2}}.
\end{split}
\end{equation}

\noindent Thus, combining \eqref{new est f}, \eqref{est norm diff A}, \eqref{diff B}, \eqref{est A} and \eqref{est lamb min}, we deduce
\begin{equation}\label{est f final }
\begin{split}
f(x_\varepsilon, u(x_\varepsilon), \eta) &\leq C|\eta|^{p(x)-1}|\log|\eta|||x-x_\varepsilon| + C|\eta|^{s-1}\log^{2}|\eta||x-x_\varepsilon| \\ & + C\varepsilon^{1-q}\dfrac{\left(|\log|\eta|||\eta|^{s-2}+|\eta|^{p(x_\varepsilon)-2}\right)^{2}}{\min\left\lbrace 1, p^{-}-1\right\rbrace \left[ |\eta|^{\dfrac{p(x)-2}{2}} + |\eta|^{\dfrac{p(x_\varepsilon)-2}{2}} \right]^{2}}|x-x_\varepsilon|^{q} \\ &  -\text{tr}\left(A(x, \eta)Z\right) -B(x, \eta).
\end{split}
\end{equation}
By Lemma \ref{propinfconv} ((i) and (vi)),  and since $u$ is locally Lipschitz,  $|x-x_\varepsilon|= O(\varepsilon)$. Hence,\begin{equation}\label{eta}
|\eta|=q\varepsilon^{1-q}|x-x_\varepsilon|^{q-1} \leq C.
\end{equation}Consequently, the terms
\begin{equation*}
|\eta|^{p(x)-1}|\log|\eta|| \quad \text{and }\quad |\eta|^{s-1}\log^{2}|\eta|
\end{equation*}remain bounded, and so the first two terms in \eqref{est f final } tend to $0$ as $\varepsilon \to 0^{+}$. For the third term, we obtain by \eqref{eta} that
\begin{equation*}
\begin{split}
\varepsilon^{1-q}\left( \dfrac{|\log|\eta||\eta|^{s-2}|}{|\eta|^{\dfrac{p(x)-2}{2}} + |\eta|^{\dfrac{p(x_\varepsilon)-2}{2}}}\right)^{2}|x-x_\varepsilon|^{q}& \leq C \log^{2}|\eta||\eta|^{2s-p(x)-2}|\eta||x-x_\varepsilon| \\
&\leq C\varepsilon \log^{2}|\eta||\eta|^{2s-p(x)-1}.
\end{split}
\end{equation*}
Since $2s-p(x)-1 \to p(x)-1 \geq p^{-}-1 >0$, the term $\log^{2}|\eta||\eta|^{2s-p(x)-1}$ is uniformly bounded for $\varepsilon$ sufficiently small and thus 
\begin{equation*}
\varepsilon^{1-q}\left( \dfrac{|\log|\eta||\eta|^{s-2}|}{|\eta|^{\dfrac{p(x)-2}{2}} + |\eta|^{\dfrac{p(x_\varepsilon)-2}{2}}}\right)^{2}|x-x_\varepsilon|^{q}  \to 0 \quad \text{as } \varepsilon \to 0^{+}.
\end{equation*}
Regarding the term
$$\varepsilon^{1-q}\dfrac{\left(|\eta|^{p(x_\varepsilon)-2}\right)^{2}}{\min\left\lbrace 1, p^{-}-1\right\rbrace \left[ |\eta|^{\dfrac{p(x)-2}{2}} + |\eta|^{\dfrac{p(x_\varepsilon)-2}{2}} \right]^{2}}|x-x_\varepsilon|^{q},$$
by \eqref{eta} it may be bounded by
$$\varepsilon^{1-q}|\eta|^{p(x_\varepsilon)-2}|x-x_\varepsilon|^{q}=q^{-1}|\eta|^{p(x_\varepsilon)-1}|x-x_\varepsilon| \to 0$$as $\varepsilon \to 0^{+}$. Therefore, we get from \eqref{est f final } that
$$f(x_\varepsilon, u_\varepsilon(x), \eta)=f(x_\varepsilon, u(x_\varepsilon), \eta) \leq -\text{tr}\left(A(x, \eta)Z\right) -B(x, \eta) + E(\varepsilon),$$with $E(\varepsilon) \to 0$ as $\varepsilon \to 0^{+}$. Thus, $u_\varepsilon$ is a viscosity supersolution of \eqref{eq for inf conv} and, since it is twice differentiable almost everywhere, \eqref{eq for inf conv} holds a. e. in $\Omega_{r(\varepsilon)}$.
\end{proof}

\begin{remark}\label{cont u}For $p^{+}< 2$, the locally Lipschitz assumption of $u$ in Lemma \ref{lemma lips} may be replaced by uniform continuity of $u$. Indeed, observe that the Lipschitz condition on $u$ was used to prove the convergences of the right-hand side terms in \eqref{est f final }. If $u$ is merely uniformly continuous,   there exists a modulus of continuity $\omega$ so that $|u(x)-u(y)| \leq \omega(x-y)$ for all $x, y \in \Omega$. Hence, by Lemma \ref{propinfconv} ((i) and (vi)), we get
$$|x_\varepsilon - x| \leq C\varepsilon^{\frac{q-1}{q}}\omega(r(\varepsilon))^{1/q},$$and then
$$|\eta|^{p(x)-1}|x_\varepsilon-x| \leq C\varepsilon^{\beta},$$where 
$$\beta:=\left( 1-\frac{1}{q}\right)(2-p^{+}) > 0.$$Therefore, the  first three terms on the right-hand side of \eqref{est f final } tend to $0$ as $\varepsilon \to 0$. 
\end{remark}

Let us pass now from the pointwise formulation in Lemma \ref{eq for inf conv} to a weak inequality. We will use in the proof some computations from the proof of \cite[Lemma 5.5]{Sil}.
\begin{lemma}\label{eq for u epsilon} Assume $f = f (x, t, \eta)$ uniformly continuous in
$\Omega\times\R\times\R^n$ and Lipschitz continuous in $\eta$, satisfying \eqref{growth}. Suppose in addition that $f(x, r, 0)=0$ for all $(x, r)\in \Omega \times \mathbb{R}$.  If $u$ is a locally Lipschitz viscosity solution of \eqref{Eq 1}, then for any non-negative  $\varphi \in \mathcal{C}_0^{\infty}(\Omega)$ there holds
\begin{equation*}
\int_{\Omega_{r(\varepsilon)}} |Du_\varepsilon|^{p(x)-2}Du_\varepsilon \cdot D\varphi\,dx \geq \int_{\Omega_{r(\varepsilon)}} f_\varepsilon(x, u_\varepsilon, Du_\varepsilon) \varphi\,dx + E(\varepsilon)\int_{\Omega_{r(\varepsilon)} \setminus \left\lbrace Du_\varepsilon=0 \right\rbrace}\varphi\,dx,
\end{equation*}for all $\varepsilon > 0$ small enough.
\end{lemma}
\begin{proof}Let $\varphi \in \mathcal{C}_0^{\infty}(\Omega)$, $\varphi\gneq 0$. Let $\varepsilon$ be small enough so that $\varphi \in \mathcal{C}_0^{\infty}(\Omega_{r(\varepsilon)})$. Since $u_\varepsilon$ is semi-concave, there is a constant $C(q, \varepsilon, u)> 0$ so that
$$\phi(x):= u_\varepsilon(x)-C(q, \varepsilon, u)|x|^{2}$$ is concave in $\Omega_{r(\varepsilon)}$.  Hence, by mollification, there is a sequence of smooth  concave functions $\phi_j$ so that 
$$(\phi_j, D\phi_j, D^{2}\phi_j) \to (\phi, D\phi, D^{2}\phi) \quad a.e. \,\,\Omega_{r(\varepsilon)}.$$Define
$$u_{\varepsilon, j}(x):=\phi_j (x)+ C(q, \varepsilon, u)|x|^{2}.$$
Given $\delta > 0$, by integration by parts we obtain
\begin{equation}\label{Parts0}
\begin{split}
\int_{\Omega_{r(\varepsilon)}} -\text{div }\left[ \left(\delta+ |Du_{\varepsilon, j}|^{2}\right)^{\frac{p(x)-2}{2}}Du_{\varepsilon, j}\right] \varphi\,dx =\int_{\Omega_{r(\varepsilon)}} \left(\delta+ |Du_{\varepsilon, j}|^{2}\right)^{\frac{p(x)-2}{2}}Du_{\varepsilon, j} \cdot D\varphi\,dx.
\end{split}
\end{equation}Observe that, since $u_\varepsilon$ is locally Lipschitz, there exists a constant $M>0$, independent of $j$, so that
\begin{equation}\label{bounds}
\sup_{j}\|D u_{\varepsilon, j}\|_{L^{\infty}(supp\,\varphi)}, \sup_{j}\|D p_{j}\|_{L^{\infty}(supp\,\varphi)} \leq M.
\end{equation}Hence, by the Dominated Convergence Theorem, the right-hand side of \eqref{Parts0} converges, as $j \to \infty$, to 
$$\int_{\Omega_{r(\varepsilon)}} (\delta+ |Du_\varepsilon|^{2})^{\frac{p(x)-2}{2}}Du_\varepsilon \cdot D\varphi\,dx. $$
Let us treat now the left-hand side of \eqref{Parts0}.  Observe that
\begin{equation}\label{Parts}
\begin{split}
&\int_{\Omega_{r(\varepsilon)}} -\text{div }\left[ \left(\delta+ |Du_{\varepsilon, j}|^{2}\right)^{\frac{p_j(x)-2}{2}}Du_{\varepsilon, j}\right] \varphi\,dx \\ & \quad =- \int_{\Omega_{r(\varepsilon)}}\left(\delta+ |Du_{\varepsilon, j}|^{2}\right)^{\frac{p_j(x)-2}{2}}\left( \Delta u_{\varepsilon, j} + \dfrac{p_j(x)-2}{\delta + |Du_{\varepsilon, j}|^{2}}\Delta_\infty u_{\varepsilon, j}\right)\varphi\,dx \\ & \qquad-\frac{1}{2}\int_{\Omega_{r(\varepsilon)}}\left(\delta+ |Du_{\varepsilon, j}|^{2}\right)^{\frac{p_j(x)-2}{2}}\log(\delta + |Du_{\varepsilon, j}|^{2})Du_{\varepsilon, j}\cdot Dp_j \varphi \,dx \\ & \qquad =: I_1 + I_2.
\end{split}
\end{equation}By \eqref{bounds} and  the Dominated Convergence 	Theorem we obtain that, when $j \to \infty$,
$$I_2 \to -\frac{1}{2}\int_{\Omega_{r(\varepsilon)}}\left(\delta+ |Du_{\varepsilon}|^{2}\right)^{\frac{p(x)-2}{2}}\log(\delta + |Du_{\varepsilon}|^{2})Du_{\varepsilon}\cdot Dp\, \varphi \,dx. $$For $I_1$ we will use Fatou's Lemma. Observe that by concavity of $\phi_j$,
$$D^{2}u_{\varepsilon, j} \leq C(q, \varepsilon, u)I.$$Hence, the integrand in $I_1$ is bounded from below by a constant independent of $j$ if $Du_{\varepsilon, j}=0$. On the other hand, if  $Du_{\varepsilon, j} \neq 0$, it can be checked (see \cite[Lemma 5.5]{Sil}) that
$$\left(\delta+ |Du_{\varepsilon, j}|^{2}\right)^{\frac{p_j(x)-2}{2}}\left( \Delta u_{\varepsilon, j} + \dfrac{p_j(x)-2}{\delta + |Du_{\varepsilon, j}|^{2}}\Delta_\infty u_{\varepsilon, j}\right) \leq C(\varepsilon, q, u, M, \delta).$$Taking $\liminf$  as $j \to \infty$ in \eqref{Parts}, we obtain
\begin{equation}\label{Parts1}
\begin{split}
&- \int_{\Omega_{r(\varepsilon)}}\left(\delta+ |Du_{\varepsilon}|^{2}\right)^{\frac{p(x)-2}{2}}\left( \Delta u_{\varepsilon} + \dfrac{p(x)-2}{\delta + |Du_{\varepsilon}|^{2}}\Delta_\infty u_{\varepsilon}\right)\varphi\,dx \\ & \qquad -\frac{1}{2}\int_{\Omega_{r(\varepsilon)}}\left(\delta+ |Du_{\varepsilon}|^{2}\right)^{\frac{p(x)-2}{2}}\log(\delta + |Du_{\varepsilon}|^{2})Du_{\varepsilon}\cdot Dp\,\varphi \,dx \\ & \leq \int_{\Omega_{r(\varepsilon)}} (\delta + |Du_\varepsilon|^{2})^{\frac{p(x)-2}{2}}Du_\varepsilon \cdot D\varphi\,dx. 
\end{split}
\end{equation}By the Dominated Convergence Theorem, as $\delta \to 0$ we have
\begin{equation}\begin{split}\label{log with delta}
\int_{\Omega_{r(\varepsilon)}}&\left(\delta+ |Du_{\varepsilon}|^{2}\right)^{\frac{p(x)-2}{2}}\log(\delta + |Du_{\varepsilon}|^{2})Du_{\varepsilon}\cdot Dp\,\varphi \,dx  \to 2\int_{\Omega_{r(\varepsilon)}}|Du_{\varepsilon}|^{p(x)-2}\log|Du_{\varepsilon}|Du_{\varepsilon}\cdot Dp\,\varphi \,dx 
\end{split}\end{equation}and
\begin{equation}\label{rhs}
\int_{\Omega_{r(\varepsilon)}} (\delta+ |Du_\varepsilon|^{2})^{p(x)-2}Du_\varepsilon \cdot D\varphi\,dx \to \int_{\Omega_{r(\varepsilon)}}  |Du_\varepsilon|^{p(x)-2}Du_\varepsilon \cdot D\varphi\,dx.
\end{equation}Moreover, using \eqref{q} and proceeding as in the proof of \cite[Lemma 5.5]{Sil}, we can apply Fatou's lemma in the integral
$$ \int_{\Omega_{r(\varepsilon)} \setminus \left\lbrace Du_\varepsilon =0 \right\rbrace}\left(\delta+ |Du_{\varepsilon}|^{2}\right)^{\frac{p(x)-2}{2}}\left( \Delta u_{\varepsilon} + \dfrac{p(x)-2}{\delta + |Du_{\varepsilon}|^{2}}\Delta_\infty u_{\varepsilon}\right)\,dx. $$
Thus, from \eqref{Parts1}, \eqref{log with delta}, \eqref{rhs}, and Lemma \ref{lemma lips} we conclude that 
\begin{equation*}
\begin{split}
& \int_{\Omega_{r(\varepsilon)}}  |Du_\varepsilon|^{p(x)-2}Du_\varepsilon \cdot D\varphi\,dx \\  
& \qquad \geq \liminf_{\delta \to 0} \int_{\Omega_{r(\varepsilon)} \setminus \left\lbrace Du_\varepsilon =0 \right\rbrace}-\left(\delta+ |Du_{\varepsilon}|^{2}\right)^{\frac{p(x)-2}{2}}\left( \Delta u_{\varepsilon} + \dfrac{p(x)-2}{\delta + |Du_{\varepsilon}|^{2}}\Delta_\infty u_{\varepsilon}\right)\varphi\,dx \\ & \qquad + \liminf_{\delta \to 0}\int_{\Omega_{r(\varepsilon)} \setminus \left\lbrace Du_\varepsilon =0 \right\rbrace}- \frac{1}{2}\left(\delta+ |Du_{\varepsilon}|^{2}\right)^{\frac{p(x)-2}{2}}\log(\delta + |Du_{\varepsilon}|^{2})Du_{\varepsilon}\cdot Dp\,\varphi \,dx \\ & \qquad \geq -\int_{\Omega_{r(\varepsilon)} \setminus \left\lbrace Du_\varepsilon =0 \right\rbrace}\Delta_{p(x)}u_\varepsilon\varphi \,dx \\ & \qquad  \geq\int_{\Omega_{r(\varepsilon)} \setminus \left\lbrace Du_\varepsilon =0 \right\rbrace} f_\varepsilon(x, u_\varepsilon, Du_\varepsilon) \varphi\,dx + E(\varepsilon)\int_{\Omega_{r(\varepsilon)} \setminus \left\lbrace Du_\varepsilon =0 \right\rbrace} |Du_\varepsilon|^{\max\left\lbrace p(x)-2, 0\right\rbrace}\varphi\,dx \\ & \qquad = \int_{\Omega_{r(\varepsilon)}} f_\varepsilon(x, u_\varepsilon, Du_\varepsilon) \varphi\,dx + E(\varepsilon)\int_{\Omega_{r(\varepsilon)}} \varphi\,dx.
\end{split}
\end{equation*}
\end{proof}

The proofs of the following lemmas follow the strategy in \cite[Lemma 5.6, Lemma 5.7]{Sil} for the homogeneous case, so we will only highlight the differences coming from the non-homogeneous term.

\begin{lemma}\label{weak convergence} Under the assumptions of Lemma \ref{eq for u epsilon}, $u \in W^{1, p(x)}_{loc}(\Omega)$ and, for each $\Omega' \Subset \Omega$, we have that, up to a subsequence, $u_\varepsilon \to u$ weakly in $W^{1, p(x)}(\Omega')$ as $\varepsilon\to 0$. 
\end{lemma}
\begin{proof}
Let $\Omega' \Subset \Omega$ and let $\xi \in\mathcal{C}^{\infty}_0(\Omega)$ be such that  $0 \leq \xi \leq 1$ and $\xi = 1$ in $\overline{\Omega'}$. Assume that 
$$K:=supp\,\xi \subset \Omega_{r(\varepsilon)},$$
and define
$$\varphi:=(L-u_\varepsilon)\xi^{p^+},\quad \mbox{with } L:=\sup_{\varepsilon, \Omega'}|u_\varepsilon(x)|.$$ By Lemma \ref{eq for u epsilon}, we have
\begin{equation}\label{inq 0 }
\begin{split}
&\int_{\Omega_{r(\varepsilon)}}|Du_\varepsilon|^{p(x)}\xi^{p^+}dx  \leq \int_{\Omega_{r(\varepsilon)}}|Du_\varepsilon|^{p(x)-1}\xi^{p^+-1}(L-u_\varepsilon)p^+|D\xi|\,dx \\ &  \qquad \qquad+ \int_{\Omega_{r(\varepsilon)}}|f_\varepsilon(x, u_\varepsilon, Du_\varepsilon)\varphi \,dx  + |E(\varepsilon)|\int_{\Omega_{r(\varepsilon)}}\varphi\,dx. 
\end{split}
\end{equation}By using Young's inequality it can be seen that 
\begin{equation}\label{ineq 1}
\begin{split}
\int_{\Omega_{r(\varepsilon)}}|Du_\varepsilon|^{p(x)-1}\xi^{p^+-1}(L-u_\varepsilon)p^+|D\xi|\,dx &\leq \delta \int_{\Omega_{r(\varepsilon)}}|Du_\varepsilon|^{p(x)}\xi^{p^+}dx + C(\delta, p,, L, D\xi),
\end{split}
\end{equation}
and, using \eqref{growth},
\begin{equation}\label{ineq 2}
\begin{split}
\int_{\Omega_{r(\varepsilon)}}|f_\varepsilon(x,& u_\varepsilon, Du_\varepsilon)|\varphi \,dx  \\
&\leq \gamma_\infty \int_{\Omega_{r(\varepsilon)}}|Du_\varepsilon|^{p(x)-1}\xi^{p^+-1}(L-u_\varepsilon)\xi\,dx + \int_{\Omega_{r(\varepsilon)}}\phi(x)(L-u_\varepsilon)\xi^{p^+}\,dx \\ & \leq \gamma_\infty \delta  \int_{\Omega_{r(\varepsilon)}}|Du_\varepsilon|^{p(x)}\xi^{p^+}+ \gamma_\infty \int_{\Omega_{r(\varepsilon)}}\left(\frac{2}{\delta}Lp^+\right)^{p(x)}\,dx+C(\phi, L, \Omega) \\ & \leq \gamma_\infty \delta  \int_{\Omega_{r(\varepsilon)}}|Du_\varepsilon|^{p(x)}\xi^{p^+}+C(\phi, \delta, p, L, \gamma, \Omega).
\end{split}
\end{equation}Finally, it is easy to check that
\begin{equation}\label{ineq 3}
\begin{split}
 |E(\varepsilon)|\int_{\Omega_{r(\varepsilon)}}\varphi\,dx &  \leq |E(\varepsilon)| C(L, p, \Omega).
\end{split}
\end{equation}Combining \eqref{inq 0 }, \eqref{ineq 1}, \eqref{ineq 2} and \eqref{ineq 3} and recalling that $\xi = 1$ in $\Omega'$, we obtain the uniform boundedness of $Du_\varepsilon$ in $L^{p(x)}(\Omega')$. Therefore, up to a subsequence, $u_\varepsilon \to u$ weakly in   $W^{1, p(x)}(\Omega')$ as $\varepsilon\to 0$.
\end{proof}

\begin{lemma}\label{strong convergence} Under the assumptions of Lemma \ref{eq for u epsilon}, for each $\Omega' \Subset \Omega$, we have that, up to a subsequence, $u_\varepsilon \to u$  in $W^{1, p(x)}(\Omega')$ as $\varepsilon\to 0$. 
\end{lemma}
\begin{proof}Let $\Omega' \Subset \Omega$ and let $\xi \in \mathcal{C}^{\infty}_0(\Omega)$ be such that  $0 \leq \xi \leq 1$ and $\xi = 1$ in $\overline{\Omega'}$. Consider the test function $$\varphi:=(u-u_\varepsilon)\xi,$$and choose $\varepsilon$ small enough so that $K:=\text{supp}\,\varphi \subset \Omega_{r(\varepsilon)}$. Observe that $\varphi \in W^{1, p(x)}(\Omega)$ and has compact support. By Lemma \ref{eq for u epsilon} and \eqref{growth}, we have
\begin{equation}\label{strong converg}
\begin{split}
\int_{\Omega_{r(\varepsilon)}} &\left(|Du|^{p(x)-2}Du - |Du_\varepsilon|^{p(x)-2}Du_\varepsilon\right)\cdot (Du - Du_\varepsilon)\xi\,dx \\ 
& \leq\int_{\Omega_{r(\varepsilon)}} |Du|^{p(x)-2}Du \cdot (Du - Du_\varepsilon)\xi\,dx + \int_{\Omega_{r(\varepsilon)}} |f_\varepsilon(x, u_\varepsilon, Du_\varepsilon)|(u-u_\varepsilon)\xi\,dx \\ 
& \quad  + |E(\varepsilon)|\int_{\Omega_{r(\varepsilon)}} (u-u_\varepsilon)\xi\,dx + \int_{\Omega_{r(\varepsilon)}} |Du_\varepsilon|^{p(x)-2}Du_\varepsilon\cdot D\xi(u-u_\varepsilon)dx \\ 
&  \leq\int_{\Omega_{r(\varepsilon)}} |Du|^{p(x)-2}Du \cdot (Du - Du_\varepsilon)\xi\,dx \\
&\quad + \|u-u_\varepsilon \|_{L^{\infty}(K)}\left(\gamma_\infty \int_K |Du_\varepsilon|^{p(x)-1} dx + \int_K \phi(x)\,dx\right) \\ 
&  \quad +\|u-u_\varepsilon \|_{L^{\infty}(K)}\left(|E(\varepsilon)||K|+\int_K |Du_\varepsilon|^{p(x)-1}D\xi dx\right).
\end{split}
\end{equation}Since $u_\varepsilon \to u$ locally uniformly and, by Lemma \ref{weak convergence}, $u_\varepsilon \to u$ weakly in $W^{1, p(x)}(K)$, the right-hand side of \eqref{strong converg} converges to $0$, up to a subsequence, when $\varepsilon \to 0$.  
By Theorem \ref{properties},  the operator $L: W^{1, p(x)}(\Omega') \to [W^{1, p(x)}(\Omega')]^{*}$ given by
$$\left\langle L(v), w\right\rangle = \int_{\Omega'}|Dv|^{p(x)-2}Dv\cdot Dw\,dx, \quad v, w \in W^{1, p(x)}(\Omega')$$is  a mapping of type ($S_{+}$). Then,  it follows from \eqref{strong converg} that $u_\varepsilon \to u$ strongly in $W^{1, p(x)}(\Omega')$ as $\varepsilon\to 0$. 
\end{proof}

We can already prove that viscosity solutions of \eqref{Eq 1} are also weak solutions.

\begin{proof}[Proof of Theorem \ref{ViscToWeak}]
Let $\varphi \in \mathcal{C}_0^{\infty}(\Omega)$ and take $\Omega' \Subset\Omega$ such that
$$\text{supp }\varphi \subset \Omega'.$$Let us fix $\varepsilon_0 > 0$ such that $0 <\varepsilon <\varepsilon_0$ implies
$$\Omega' \subset \Omega_{r(\varepsilon)}.$$
In view of Lemma \ref{eq for u epsilon}, to prove the theorem, it will be enough to show the following convergences:
\begin{itemize}
\item[(I)] $$\lim_{\varepsilon \to 0^{+}} \int_{\Omega'}|Du_\varepsilon|^{p(x)-2}Du_\varepsilon \cdot D\varphi\,dx = \int_{\Omega'}|Du|^{p(x)-2}Du\cdot D\varphi\,dx$$
\item[(II)] $$\lim_{\varepsilon \to 0^{+}} \int_{\Omega_{r(\varepsilon)}}f_\varepsilon(x, u_\varepsilon, Du_\varepsilon)\varphi\,dx = \int_{\Omega'}f(x, u, Du)\varphi \,dx$$
\item[(III)] $$\lim_{\varepsilon \to 0^{+}} E(\varepsilon)\int_{\Omega_{r(\varepsilon)}}\varphi \,dx=0.$$
\end{itemize} Proceeding exactly as in the proof of \cite[Theorem 5.8]{Sil}, it can be seen that (I) holds, and (III) follows in a straightforward way.

Let us prove (II). Let $\varepsilon$, $\varphi$ and $\Omega'$ as above. By the uniform continuity of $f$, for every $\rho >0$, 
        there exists $\delta >0$ such that
        $$
            |f(x, u_\varepsilon(x), D u_\varepsilon(x))-f(y, u_\varepsilon(x), D u_\varepsilon(x))| \leq \rho, \quad y \in B_\delta(x).
        $$Choose $\varepsilon_0>0$ so that $r(\varepsilon)<\delta$ for every $\varepsilon<\varepsilon_0$. Thus, from the previous inequality we get
$$f(x,u_\varepsilon(x),D u_\varepsilon(x))<\rho+f(y,u_\varepsilon(x),D u_\varepsilon(x)),$$
for every $x\in \Omega'$ and $y\in B_{r(\varepsilon)}(x)$. In particular,
$$f(x,u_\varepsilon(x),D u_\varepsilon(x))<\rho+f_\varepsilon(x,u_\varepsilon(x),D u_\varepsilon(x)),$$
and therefore
$$0 \leq |f(x,u_\varepsilon(x),D u_\varepsilon(x))-f_\varepsilon(x,u_\varepsilon(x),D u_\varepsilon(x))|<\rho.$$
        Hence,
        \begin{equation}\label{part 0}
            \int_{\Omega'} |f(x, u_\varepsilon, D u_\varepsilon)-f_\varepsilon(x, 
            u_\varepsilon, D u_\varepsilon)|\varphi\,dx \leq \rho 
            \|\varphi\|_{L^{\infty}(\Omega')}|\Omega'|, 
        \end{equation}for $\rho$ arbitrarily small.   
        Since $\|u_\varepsilon\|_{L^{\infty}(\Omega')} \leq \|u\|_{L^{\infty}(\Omega')}$ for all $\varepsilon$, 
        it follows
        $$
            \max_{[-\|u_\varepsilon\|_{L^{\infty}},
            \|u_\varepsilon\|_{L^{\infty}}]}|\gamma(t)| 
            \leq 
            \max_{[-\|u\|_{L^{\infty}}, \|u\|_{L^{\infty}}]}|\gamma(t)|,
        $$
        and then by \eqref{growth} we have
        $$
            |f(x, u_\varepsilon, D u)| \leq C|D u|^{p(x)-1}+\phi(x) \in L^{p'(x)}(\Omega') \subset L^{1}(\Omega')
        $$
        for a constant $C$ independent of $\varepsilon$. Then, by the Lebesgue Convergence Theorem,
         \begin{equation}\label{part i}
            \lim_{\varepsilon \to 0}\int_{\Omega'} f(x, u_\varepsilon, D u)\varphi\,dx
            = \int_{\Omega'}f(x, u, D u)\varphi\,dx.
        \end{equation}Moreover, the convergence $Du_\varepsilon \to Du$ in $L^{p(x)}(\Omega')$ and the Lipschitz continuity of $f$ with respect to the third variable imply
        \begin{equation}\label{part ii}
        \begin{split}
        \int_{\Omega'}|f(x, u_\varepsilon, Du_\varepsilon)-f(x, u_\varepsilon, Du )|\varphi\,dx \leq C\int_{\Omega'}|Du_\varepsilon-Du|\,dx \to 0 \quad \text{ as } \varepsilon \to 0.
        \end{split}
        \end{equation}Therefore, combining \eqref{part 0},  \eqref{part i}, and \eqref{part ii} we obtain (II).
        \end{proof}

\section{Weak solutions are viscosity solutions: proof of Theorem \ref{WeakToVisc}}\label{WtoV}
To prove this implication we follow the strategy of \cite{MO}. As we said in the Introduction, the argument is strongly connected to the availability of comparison principles. After the proof of the theorem we will state and prove an example of comparison result that applies here. 
\begin{proof}[Proof of Theorem \ref{WeakToVisc}]
Let $u  \in  \mathcal{C}(\Omega)$ be a weak supersolution to \eqref{Eq 1}. To reach a contradiction, assume that $u$ is not a viscosity supersolution.  By assumption, there exist $x_0 \in \Omega$ and $\varphi \in \mathcal{C}^{2}(\Omega)$ so that $D \varphi (x_0) \neq 0$, 
\begin{equation}\label{minn}
u(x_0)=\varphi(x_0), \quad u(x)>\varphi(x) \mbox{ for all $x\neq x_0$},
\end{equation}and
\begin{equation}\label{minn 2}
-\Delta_{p(x_0)} \varphi(x_0) < f(x_0, u(x_0), D \varphi(x_0)).
\end{equation}
Moreover, the mapping
$$x \to f(x, u(x), D \varphi(x))$$is continuous in $\Omega$, and \eqref{minn} yields

$$-\Delta_{p(x)} \varphi(x)-f(x,u(x),D\varphi(x))<0,\quad \mbox{ for all }x\in B_r(x_0),$$
for some $r>0$ small eunogh. Hence, there exists $r_0>0$ so that $D \varphi(x) \neq 0$ for all $x \in B_{r_0}(x_0)$ and
\begin{equation}\label{eqq}
-\Delta_{p(x)}\varphi(x)\leq f(x,u(x),D \varphi(x)),\qquad x\in B_{r_0}(x_0).
\end{equation}Let
$$m:=\min_{\partial B_{r_0}(x_0)}\left( u-\varphi\right).$$Then by \eqref{minn}, $m >0$. Consider
$$\tilde{\varphi}(x):=\varphi(x) +m, \quad x \in \Omega.$$
By \eqref{eqq},  $\tilde{\varphi}$ is a weak subsolution to
\begin{equation}\label{new eq}
-\Delta_{p(x)} v = \tilde{f}(x, D v),
\end{equation}in $ B_{r_0}(x_0)$, where $\tilde{f}(x, \eta):=f(x, u(x), \eta)$. Observe that $\tilde{f}$ is  locally Lipschitz in $\Omega \times \mathbb{R}^{n}$. Moreover, in the weak sense, we have
\begin{equation*}
-\Delta_{p(x)} u  \geq f(x, u, D u) = \tilde{f}(x, D u),
\end{equation*}which shows that $u$ is a weak supersolution to \eqref{new eq}. In addition, observe that $u \geq \tilde{\varphi}$ on $\partial  B_{r_0}(x_0)$, and that \eqref{assump grad} holds since $D \tilde{\varphi} \neq 0$ in $ B_{r_0}(x_0)$.  Thus, by the (CPP) we conclude that $u \geq \tilde{\varphi}$ in $B_{r_0}(x_0)$. This contradicts \eqref{minn}.
\end{proof}

\subsection{Maximum principles}
Some comparison principles may be found in the literature for $p(x)$-Laplace equations.  See for instance \cite{FZ0, JLP, TG}. Here, we provide a comparison principle for a Lipschitz right-hand side depending on all the lower terms. 

\begin{theorem}\label{comparison} Assume that $f=f(x, r, \eta)$ is locally Lipschitz in $\Omega \times \mathbb{R}\times \mathbb{R}^{n}$. Let $u, v \in \mathcal{C}^{1}(\Omega)$ be weak sub- and supersolutions, respectively,  of \eqref{Eq 1} such that
\begin{equation}\label{assump grad}
|D u(x)|+ |D v(x)| >0,
\end{equation}for a. e. $x$ satisfying  $p(x)> 2$.  Then, there exists $\delta > 0$ such that for any domain $B$,  $\overline{B} \subset \Omega$,  with  $|B|< \delta$ and $u \leq v$ on $\partial B$, there holds $u \leq v$ in $B$.
\end{theorem}
\begin{proof}
Take $(u-v)^{+}\chi_B \in W_0^{1, p(x)}(\Omega)$ as a test function in \eqref{Eq 1} to get
\begin{equation}
\begin{split}
&\int_B \left( |D u|^{p(x)-2}D u - |D v|^{p(x)-2}D v \right)\cdot D (u-v)^{+} \\ &  \qquad \leq \int_B \frac{f(x, u, D u)-f(x, v, D v)}{u-v}[(u-v)^{+}]^{2}.
\end{split}
\end{equation}Now, using the inequality
$$c(|\xi|+|\eta|)^{p(x)-2}|\xi-\eta|^{2}\leq \left( |\xi|^{p(x)-2}\xi-|\eta|^{p(x)-2}\eta\right)\cdot ( \xi-\eta),$$ the Lipschitz assumption on $f$ and the boundedness of $u$ and $v$ in $\mathcal{C}^{1}$, we get
\begin{equation}\label{com 1}
\begin{split}
&\int_B (|D u|+ |D v|)^{p(x)-2}|D (u-v)^{+}|^{2}  \leq C(u, v) \left[\int_B [(u-v)^{+}]^{2}+ \int_B |D u-D v|(u-v)^{+}\right]. 
\end{split}
\end{equation}By  Poincar\'e and H\"{o}lder inequalities, and the assumption \eqref{assump grad}, we obtain
\begin{equation}\label{comb 2}
\begin{split}
C(u, &v)\left[\int_B [(u-v)^{+}]^{2}+ \int_B |D u-D v|(u-v)^{+}\right] \\ & \leq C(B)C(u, v)\int_B |D (u-v)^{+}|^{2} \quad (\textnormal{here, }C(B) \to 0 \,\,\textnormal{as }|B|\to 0) \\ &  = C(B)C(u, v)\int_B  (|D u|+ |D v|)^{2-p(x)}(|D u|+ |D v|)^{p(x)-2}|D (u-v)^{+}|^{2} \\ & \leq C(B)C(u, v)\int_B   (|D u|+ |D v|)^{p(x)-2}|D (u-v)^{+}|^{2}. 
\end{split}
\end{equation}Combining \eqref{com 1} and \eqref{comb 2}, we obtain
$$\int_B (|D u|+ |D v|)^{p(x)-2}|D (u-v)^{+}|^{2} \leq C(B)C(u, v)\int_B   (|D u|+ |D v|)^{p(x)-2}|D (u-v)^{+}|^{2}.$$Hence, for $|B|$ small enough, we get $(u-v)^{+}=0$  in $B$ and hence $u \leq v$  in $B$.
\end{proof}

\begin{remark}
Keeping track of the proof of Theorem \ref{WeakToVisc}, it is easy to see that Theorem \ref{comparison} allows us to prove that weak solutions are viscosity. Indeed, it is enough with choosing $r_0$ sufficiently small so that $|B_{r_0}(x_0)| < \delta$, with $\delta >0$ provided by Theorem \ref{comparison}.  
\end{remark}

\section*{Acknowledgements}
M.~Medina  has been supported by Project PDI2019-110712GB-100, MICINN, Spain. P. Ochoa has been supported by CONICET and Grant B080, UNCUYO, Argentina.

\end{document}